\documentclass[12pt,reqno]{amsart}

\usepackage[a4paper,bindingoffset=0.1in,left=1.2in,right=1.3in,top=1.4in,bottom=1.5in,footskip=.25in]{geometry}

\usepackage{indentfirst,amssymb,amsmath,amsthm}    
\usepackage{newtxtext,newtxmath}
\usepackage{setspace}
\usepackage{times}
\usepackage[utf8]{inputenc}
\usepackage[T1]{fontenc}
\usepackage{verbatim}

\usepackage{hyperref}
\hypersetup{colorlinks=true, linkcolor=blue,citecolor=blue, urlcolor=blue}
\DeclareMathOperator{\pideg}{PI-deg}

\DeclareMathOperator{\cf}{Fract}

\DeclareMathOperator{\de}{det}
\DeclareMathOperator{\ke}{ker}
\DeclareMathOperator{\tors}{tor}

\DeclareMathOperator{\spa}{span}
\DeclareMathOperator{\mo}{mod}
\DeclareMathOperator{\mi}{min}

\DeclareMathOperator{\tr}{tr}

\numberwithin{equation}{section}

\newtheorem{theo}{Theorem}[section]
\newtheorem{theom}{Theorem}[subsection]

\newtheorem{lemm}{Lemma}[section]
\newtheorem{lemma}{Lemma}[subsection]
\newtheorem{rema}{Remark}[section]
\newtheorem{remak}{Remark}[subsection]
\newtheorem{coro}{Corollary}[section]
\newtheorem{corol}{Corollary}[subsection]
\newtheorem{prop}{Proposition}[section]
\newtheorem{propn}{Proposition}[subsection]

\newtheorem*{theorem A}{Theorem A}
\newtheorem*{theorem B}{Theorem B}

\begin{document}

\setcounter{page}{1} 
\baselineskip .65cm 
\pagenumbering{arabic}

\title[Quantized Matrix Algebras at a root of unity]{Dipper Donkin Quantized Matrix Algebra\\ and Reflection Equation Algebra at root of unity}
\author [Snehashis Mukherjee~And~Sanu Bera]{Snehashis Mukherjee$^1$ \and Sanu Bera$^2$}

\address {\newline Snehashis Mukherjee$^1$~~and Sanu Bera$^2$\newline School of Mathematical Sciences, \newline Ramakrishna Mission Vivekananda Educational and Research Institute (rkmveri), \newline Belur Math, Howrah, Box: 711202, West Bengal, India.
 }
\email{\href{mailto:tutunsnehashis@gmail.com}{tutunsnehashis@gmail.com$^1$};\href{mailto:sanubera6575@gmail.com}{sanubera6575@gmail.com$^2$}}

\subjclass[2020]{16D60, 16T20, 16S36, 16R20}
\keywords{simple modules, quantized matrix algebras, polynomial identity algebra.}
\begin{abstract}
In this article the quantized matrix algebras as in the title have been studied at a root of unity. A full classification of simple modules over such quantized matrix algebras of rank $2$ along with a class of finite dimensional indecomposable modules are given.
\end{abstract}
\maketitle

\section{\bf{Introduction}}\label{sec1}
Let $\mathbb{K}$ be a field and an element $q\in \mathbb{K}^*$. Many $\mathbb{K}$-algebras appearing in representation theory exhibit interesting $q$-deformations-families of algebras depending on a nonzero scalar parameter $q$, whose value at $q=1$ is the original algebra. Let us consider the matrix algebra $M_n(\mathbb{K})$. There are several ways to quantized the matrix algebra $M_n(\mathbb{K})$. Among them most common is probably the algebra introduced by Faddeev, Reshetikhin, and Takhtajan in \cite{frt}. Denote that algebra by $M_n(q)$. In \cite{dd}, Dipper and Donkin defined another quantized matrix algebra $\mathcal{M}\text{at}_n(q)$ which has many features which are different from $M_n(q)$ e.g., quantum determinant is not central. In \cite{mt}, a two parameter quantized matrix algebra was defined for which the above two quantized matrix algebras are special cases. 
\subsection{Dipper-Donkin Quantized Matrix Algebra} The Dipper-Donkin quantized matrix algebra $\mathcal{M}\text{at}_n(q)$ of rank $n$ is an associative algebra over $\mathbb{K}$  generated by elements $Z_{ij},~1\leq i,j\leq n$ subject to the following relations:
\[\begin{array}{l}
Z_{ij}Z_{st}=qZ_{st}Z_{ij}, \ \text{if}\  i>s,j\leq t,\\
Z_{ij}Z_{st}=Z_{st}Z_{ij}+(q-1)Z_{sj}Z_{it}, \ \text{if}\ i>s,j>t\\
Z_{ij}Z_{ik}=Z_{ik}Z_{ij}, \ \text{for~ all}\ i,j,k.
\end{array}\]
In fact the algebra $\mathcal{M}\text{at}_n(q)$ is a bialgebra and its structure maps are as follows:
\[\Delta(Z_{ij})=\sum\limits_{k}Z_{ik}\otimes Z_{kj},\ \epsilon(Z_{ij})=\delta_{ij},\ \ \forall\ 1\leq i,j\leq n\]
where $\delta_{ij}$ is the Kronecker delta, which equals $1\in\mathbb{K}$ when $i=j$ and otherwise equals $0\in\mathbb{K}$. In \cite{jo}, Jakobsen and Zhang computed explicitly the center and the PI degree of the algebra $\mathcal{M}\text{at}_n(q)$ if $q$ is a primitive $m$-th root of unity and characteristic of $\mathbb{K}$ is zero.
In this article we shall focus on the representations of the Dipper-Donkin quantized matrix algebra of rank $2$. 
\subsection{Reflection Equation Algebra} 
The ``reflection equation'' was first introduced by Cherendik in his study \cite{ch} of factorizable scattering on a half-line, and reflection equation algebras later emerged from Majid’s transmutation theory in \cite{mj}. The reflection equation algebra $A_q(M_2)$ of rank $2$ is an associative algebra over $\mathbb{K}$ generated by $u_{11},\ u_{12},\ u_{21},\ u_{22}$ together with the relations  
\[\begin{array}{c}
u_{11}u_{22}=u_{22}u_{11},\ u_{22}u_{12}=q^2u_{12}u_{22},\ u_{21}u_{22}=q^2u_{22}u_{21}\\ u_{11}u_{12}=u_{12}(u_{11}+(q^{-2}-1)u_{22})\\
u_{21}u_{11}=(u_{11}+(q^{-2}-1)u_{22})u_{21}\\
u_{21}u_{12}-u_{12}u_{21}=(q^{-2}-1)u_{22}(u_{22}-u_{11})
\end{array} \]
In \cite{ku}, Kulish and Sklyanin show that $A_q(M_2)$ has a $\mathbb{K}$-basis consisting of monomials in the generators $u_{ij}$. They compute the center of $A_q(M_2)$. Domokos and Lenagan address $A_q(M_n)$ for general $n$ in \cite{dl}. They show that $A_q(M_n)$ is a noetherian domain, and that it
has a $\mathbb{K}$-basis consisting of monomials in the generators $u_{ij}$. In \cite{eb}, the simple finite dimensional $A_q(M_2)$-modules were classified, finite dimensional weight modules were shown to be semisimple, Aut($ A_q(M_2)$) was computed, and the prime spectrum of $A_q(M_2)$ was computed along with its Zariski topology by E. Ebrahim. Finally, it was shown that $A_q(M_2)$ satisfies the Dixmier-Moeglin equivalence assuming $q$ is a not a root of unity.\\
\textbf{Aim:} In this article our aim is to classify simple modules over the quantized matrix algebras $\mathcal{M}\text{at}_2(q)$ and $ A_q(M_2)$ respectively assuming that $q$ is a primitive $m$-th root of unity and $\mathbb{K}$ is an algebraically closed field. \\
\textbf{Arrangement:} The paper is organized as follows: In Section $2$, we will recall some known facts for $\mathcal{M}\text{at}_n(q)$ and then discuss about the theory of Polynomial Identity algebras to comment about the $\mathbb{K}$-dimension of the simple modules over such quantized matrix algebras. Also an explicit expression of PI-degree for $\mathcal{M}\text{at}_n(q)$ is established with the help of derivation erasing process and PI-deg of quantum affine space. In Section $3$ we classify all simple $\mathcal{M}\text{at}_2(q)$-modules. In section $4$, we wish to construct some finite dimensional indecomposable modules over $\mathcal{M}\text{at}_2(q)$. In Section $5$, we focus on the algebra $A_q(M_2)$ at root of unity. In Section $5$ and $6$ we classify simple $A_q(M_2)$-modules and construct some indecomposable $A_q(M_2)$-modules respectively.
\section{\bf{Preliminaries}}
Till the end of the paper $\mathbb{K}$ is an algebraically closed field and all modules are right modules.
\subsection{Torsion and Torsion free modules:} 
Let $A$ be an algebra and $M$ be a right $A$-module and $S\subset A$ be a right Ore set. The submodule
\[\tors_{S}(M):=\{m\in M|ms=0\ \text{for some}\  s\in S\}\]
is called the $S$-torsion submodule of $M$. The module $M$ is said to be $S$-torsion if $\tors_{S}(M)=M$ and $S$-torsion free if $\tors_{S}(M)=0$. If the Ore set $S$ is generated by $x\in A$, we simply say that the $S$-torsion/torsionfree module $M$ is $x$-torsion/torsionfree.
\par A non zero element $x$ of an algebra ${A}$ is called a normal element if $x{A}={A}x$. Clearly if $x$ is a normal element of $A$, then the set $\{x^i~|~i\geq 0\}$ is an Ore set generated by $x$. The next lemma is obvious.
\begin{lemma}\label{itn}
Suppose that $A$ is an algebra, $x$ is a normal element of $A$ and $M$ is a simple $A$-module. Then either $Mx=0$ (if $M$ is $x$-torsion) or the map $x_{M}:M\rightarrow M, m\mapsto mx$ is an isomorphism (if $M$ is $x$-torsion free).
\end{lemma}
The above lemma says that the action of a normal element on a simple module is either trivial or invertible.
\subsection{Facts about $\mathcal{M}\text{at}_2(q)$:} First recall that $\mathcal{M}\text{at}_2(q)$ is an associative $\mathbb{K}$-algebra generated by $Z_{11},\ Z_{12},\ Z_{21},\ Z_{22}$ together with the relations  
\[\begin{array}{ll} 
Z_{12}Z_{11}=Z_{11}Z_{12},& Z_{21}Z_{11}=qZ_{11}Z_{21}\\
Z_{22}Z_{21}=Z_{21}Z_{22},&Z_{22}Z_{12}=qZ_{12}Z_{22}\\
Z_{21}Z_{12}=qZ_{12}Z_{21},&Z_{22}Z_{11}=Z_{11}Z_{22}+(q-1)Z_{12}Z_{21}.
\end{array}\]
The algebra $\mathcal{M}\text{at}_2(q)$ has an iterated skew polynomial presentation with respect to the ordering of variables $Z_{11},Z_{12},Z_{21},Z_{22}$ of the form: \[\mathbb{K}[Z_{11}][Z_{12},\rho_1][Z_{21},\rho_2][Z_{22},\sigma,\delta]\]
where the $\rho_1,\rho_2$ and $\sigma$ are $\mathbb{K}$-linear automorphisms and the $\delta$ is $\mathbb{K}$-linear $\sigma$-derivation such that
\[
\rho_1(Z_{11})=Z_{12},\rho_2(Z_{11})=qZ_{11},\rho_2(Z_{12})=qZ_{12},\]
\[\sigma(Z_{11})=Z_{11},~\sigma(Z_{12})=qZ_{12},~\sigma(Z_{21})=Z_{21},\]
\[\delta(Z_{11})=(q-1)Z_{12}Z_{21},~\delta(Z_{12})=\delta(Z_{21})=0.\]
This observation along with the skew polynomial version of the Hilbert Basis Theorem (cf. \cite[Theorem 2.9]{mcr}) yields the following proposition.
\begin{propn}\label{imp}
The algebra $\mathcal{M}\text{at}_2(q)$ is an affine noetherian domain. Moreover, the family $\{Z_{11}^aZ_{12}^bZ_{21}^cZ_{22}^d~|~a,b,c,d\geq 0\}$ is a $\mathbb{K}$-basis of $\mathcal{M}\text{at}_2(q)$.
\end{propn}
Let us consider an element $\de_{q}=Z_{11}Z_{22}-Z_{12}Z_{21}$ of the algebra $\mathcal{M}\text{at}_2(q)$ which is known as quantum determinant of $\mathcal{M}\text{at}_2(q)$ (cf. \cite{jo}). Then one can easily verify the following relations: 
\[\de_{q}=Z_{22}Z_{11}-qZ_{12}Z_{21},\]
\[\de_{q} \ Z_{11}=Z_{11}\ \de_q,\ \ \de_{q} \ Z_{12}=qZ_{12} \ \de_q,\]
\[\de_{q} \ Z_{21}=q^{-1} \ Z_{21}  \ \de_q,\ \ \de_{q} \ Z_{22}=Z_{22} \ \de_q.\]
So the element $\det_{q}$ is a normal element of the algebra $\mathcal{M}\text{at}_2(q)$. Also from the defining relations of the algebra $\mathcal{M}\text{at}_2(q)$, it follows that $Z_{12}$ and $Z_{21}$ are normal elements of $\mathcal{M}\text{at}_2(q)$.
\begin{lemma}
\emph{(\cite[Lemma 4.2]{jo})}\label{c1}
For $n\geq 1$, the following identities hold in the algebra $\mathcal{M}\text{at}_2(q)$:
\begin{enumerate}
    \item [(i)] $Z_{22}Z_{11}^n=Z_{11}^nZ_{22}+(1-q^{-n})Z_{21}Z_{12}Z_{11}^{n-1}$,
    \item [(ii)] $Z_{11}Z_{22}^n=Z_{22}^nZ_{11}+(1-q^{n})Z_{12}Z_{21}Z_{22}^{n-1}$.
\end{enumerate}
\end{lemma}
The equalities can be proved by induction on $n$. The defining relations of $\mathcal{M}\text{at}_2(q)$ along with this lemma we have:
\begin{corol}\label{dl1}
If $q$ is a primitive $m$-th root of unity, then $Z_{ij}^m$ is a central element for all $i,j=1,2$.
\end{corol}
All the required facts of the reflection equation algebra of rank $2$ have been mentioned in section \ref{reas}.
\subsection{Polynomial Identity Algebras}
In this subsection, we recall some known facts concerning Polynomial Identity algebra that we shall be applying on $\mathcal{M}\text{at}_2(q)$ and $A_q(M_2)$ at the root of unity for further development. First we recall a result which provides a sufficient condition for a ring to be PI. 

\begin{propn}\emph{(\cite[Corollary 13.1.13]{mcr})}\label{f}
If $R$ is a ring which is a finitely generated module over a commutative subring, then $R$ is a PI ring.
\end{propn}

\begin{propn} \label{finite}
The algebra $\mathcal{M}\text{at}_2(q)$ is a PI algebra if and only if $q$ is a root of unity.
\end{propn}
\begin{proof}
Suppose $q$ be a primitive $m$-th root of unity. Let $Z$ be the subalgebra of $\mathcal{M}\text{at}_2(q)$ generated by $Z_{ij}^m$ for all $i,j=1,2$. Then by Corollary \ref{dl1}, the subalgebra $Z$ is central. Now from the $\mathbb{K}$-basis of $\mathcal{M}\text{at}_2(q)$, one can easily verify  that $\mathcal{M}\text{at}_2(q)$ is a finitely generated module over $Z$, with basis 
\[\{Z_{11}^aZ_{12}^bZ_{21}^cZ_{22}^d~|~0\leq a,b,c,d\leq m-1\}.\]
Hence it follows from Proposition \ref{f} that $\mathcal{M}\text{at}_2(q)$ is PI algebra.
\par For the converse, just note that the $\mathbb{K}$-subalgebra of $\mathcal{M}\text{at}_2(q)$ generated by $Z_{11}$ and $Z_{21}$ with relation $Z_{21}Z_{11}=qZ_{11}Z_{21}$ is not PI if $q$ is not a root of unity (cf. \cite[Proposition I.14.2.]{brg}).
\end{proof}
\begin{remak}\label{re1}
Now the algebra $\mathcal{M}\text{at}_2(q)$ being a finite module over a central subalgebra, it follows from a standard result (cf. \cite[Proposition III.1.1.]{brg}) that every simple $\mathcal{M}\text{at}_2(q)$-module is finite dimensional vector space over $\mathbb{K}$.
\end{remak}
\par Primitive PI rings exhibit a particularly nice structure established in Kaplansky's Theorem (cf. \cite[Theorem 13.3.8]{mcr}). Now Kaplansky's Theorem has a striking consequence in case of a prime affine PI algebra over an algebraically closed field.
\begin{propn}\emph{(\cite[Theorem I.13.5]{brg})}\label{sim}
Let $A$ be a prime affine PI algebra over an algebraically closed field $\mathbb{K}$, with PI-deg($A$) = $n$ and $V$ be a simple $A$-module. Then $V$ is a vector space over $\mathbb{K}$ of dimension $t$, where $t \leq n$, and $A/ann_A(V) \cong M_t(\mathbb{K})$.
\end{propn}
The above proposition yields the important link between the PI degree of a prime affine PI algebra over an algebraically closed field and its irreducible representations. Thus from Proposition $\ref{finite}$ and Proposition $\ref{imp}$ along with Proposition $\ref{sim}$, it is quite clear that each simple $\mathcal{M}\text{at}_2(q)$-module is finite dimensional and can have dimension at most $\pideg(\mathcal{M}\text{at}_2(q))$. Infact this bound is attained by some simple $\mathcal{M}\text{at}_2(q)$-modules.
\par So far all the results as mentioned above for the algebra $\mathcal{M}\text{at}_2(q)$ is also true for $\mathcal{M}\text{at}_n(q)$ in general (cf. \cite{dd,jo}).

\subsection{PI Degree of $\mathcal{M}\text{at}_n(q)$}\label{pisub} 
Let $\mathbf{q}=(q_{ij})$ be a multiplicatively antisymmetric $(n\times n)$-matrix over $\mathbb{K}$, that is, $q_{ii}=1$ for all $i$ and $q_{ji}={q_{ij}}^{-1}$ for all $i\neq j$. Given such a matrix, the multiparameter quantum affine space of rank $n$ is the $\mathbb{K}$-algebra $\mathcal{O}_{\mathbf{q}}(\mathbb{K}^n)$ generated by the variables $x_1,\cdots ,x_n$ subject only to the relations
\begin{equation} \label{relation}
x_ix_j=q_{ij}x_jx_i, \ \ \ \forall\ \ \ 1 \leq i,j\leq n.
\end{equation}
It is of interest to know when a quantum affine space $\mathcal{O}_{\mathbf{q}}(\mathbb{K}^n)$ is a PI ring and what 
its PI degree is. Using Kaplansky's Theorem and Proposition \ref{f}, one can prove that $\mathcal{O}_{\mathbf{q}}(\mathbb{K}^n)$ is a PI ring if and only if all the $q_{ij}$ are roots of unity. The following result of De Concini and Procesi provides one of the key techniques for calculating the PI degree of a quantum affine space over arbitrary characteristic of $\mathbb{K}$.
\begin{propn}\emph{(\cite[Proposition 7.1]{dicon})}\label{quan}
Suppose that $q_{ij}=q^{h_{ij}}$ for all $i,j$, where $q \in \mathbb{K}^*$ is a primitive $m$-th root of unity and the $h_{ij} \in \mathbb{Z}$. Let $h$ be the cardinality of the image of the homomorphism 
\[
    \mathbb{Z}^n \xrightarrow{(h_{ij})} \mathbb{Z}^n \xrightarrow{\pi} \left(\mathbb{Z}/m\mathbb{Z}\right)^n,
\]
where $\pi$ denotes the canonical epimorphism. Then $\pideg(\mathcal{O}_{\mathbf{q}}(\mathbb{K}^n))=\sqrt{h}$.
\end{propn}
In \cite{jo}, Jakobsen and Zhang gave an explicit %canonical form for the defining matrix of the algebra $\mathcal{M}\text{at}_n(q)$ 
expression of the PI degree for $\mathcal{M}\text{at}_n(q)$ over the field of characteristic zero. Infact they established the following result:
\begin{theom}\emph{(\cite[Theorem 3.1]{jo})}\label{cho}
Let $\mathbb{K}$ be a field of characteristic zero. If $q$ is a primitive $m$-th root of unity, then 
\[\pideg(\mathcal{M}\text{at}_n(q))=\pideg(\mathcal{O}_{\mathbf{q}}(\mathbb{K}^{n^2}))=\begin{cases}
m^{\frac{n^2}{2}},& n \ \ \text{even}\\
m^{\frac{n^2-1}{2}}, & n \ \ \text{odd}.
\end{cases}\]
where the $n^2\times n^2$ multiplicatively antisymmetric matrix $\mathbf{q}$ is as in \cite[eq 10]{jo}.
\end{theom}
The first equality of the above theorem follows from the derivation erasing over characteristic zero in \cite[Theorem 6.4]{dicon} due to De Concini and Procesi. The second equality can be proved using the key Proposition \ref{quan} for calculating PI degree along with \cite[Lemma 5.7]{ar}.
\par Now we aim to establish the above theorem for arbitrary characteristic. For that it is enough to show that the first equality is valid for arbitrary characteristic. Here we will use another derivation erasing process (independent of characteristic) duo to Leroy and Matczuk as in \cite[Theorem 7]{lm2}. We define an ordering on the set of generators $Z_{ij}$, $1\leq i,j\leq n$ of $\mathcal{M}\text{at}_n(q)$ by lexicographic ordering on the set of indices. With this, $\mathcal{M}\text{at}_n(q)$ can be expressed as an iterated skew polynomial ring of the form:
\[\mathbb{K}[Z_{11}][Z_{12},\sigma_{12},\delta_{12}]\cdots [Z_{nn},\sigma_{nn},\delta_{nn}]\]
where the $\sigma_{ij}$ are $\mathbb{K}$-linear automorphisms and the $\delta_{ij}$ are $\mathbb{K}$-linear $\sigma_{ij}$-derivation such that for all $1\leq i,j\leq n$:
\[\sigma_{ij}(Z_{rs})=\begin{cases}
Z_{rs},r\leq i,s<j\\
qZ_{rs},r<i
\end{cases}\ \ 
\delta_{ij}(Z_{rs})=\begin{cases}
(q-1)Z_{rj}Z_{is},& r<i,s<j\\
0,&\text{otherwise}.
\end{cases}\]
One can easily verify the following skew relations for the algebra $\mathcal{M}\text{at}_n(q)$ \[\sigma_{ij}\delta_{ij}=q\delta_{ij}\sigma_{ij},~(q\neq 1),~\forall~~1\leq i,j\leq n.\]  Then the derivation erasing process (independent of characteristic) in \cite[Theorem 7]{lm2} provides $\cf {\mathcal{M}\text{at}_n(q)}\cong \cf \mathcal{O}_{\mathbf{q}}(\mathbb{K}^{n^2})$ and hence $\pideg \mathcal{M}\text{at}_n(q)=\pideg \mathcal{O}_{\mathbf{q}}(\mathbb{K}^{n^2})$, where the $(n^2\times n^2)$-matrix of relations $\mathbf{q}$ is as in Theorem \ref{cho}. Thus we conclude that Theorem \ref{cho} is true for arbitrary characteristic. In particular $\pideg(\mathcal{M}\text{at}_2(q))=m^2$.

\section{\bf{Simple Modules over $\mathcal{M}\text{at}_2(q)$}}
Let $q$ be a primitive $m$-th root of unity and $\mathcal{N}$ be a simple module over $\mathcal{M}\text{at}_2(q)$. Then by Proposition $\ref{sim}$, the $\mathbb{K}$-dimension of $\mathcal{N}$ is a finite and bounded above by $\pideg({\mathcal{M}\text{at}_2(q)})$. In this section we wish to classify simple $\mathcal{M}\text{at}_2(q)$-modules. Note that the elements $\de_{q},Z_{12}$ and $Z_{21}$ are normal, and $\de_{q}$ and $Z_{12}Z_{21}$ commuting elements in $\mathcal{M}\text{at}_2(q)$. Then there is a common eigenvector $v$ in $\mathcal{N}$ of the commuting operators $\de_{q}$ and $Z_{12}Z_{21}$. Put
\[v\de_{q}=\lambda_1v,\ vZ_{12}Z_{21}=\lambda_2v,\ \ \lambda_1,\lambda_2\in \mathbb{K}.\]
Now depending on the scalars $\lambda_1$ and $\lambda_2$ along with the Lemma \ref{itn}, we will consider the following subsections.
\subsection{Simple $\de_{q}$-torsionfree and $Z_{12}Z_{21}$-torsion $\mathcal{M}\text{at}_2(q)$-modules:}\label{ss1}
In this case $\lambda_1\neq 0$ and $\lambda_2=0$. Then $vZ_{12}Z_{21}=0$ implies either $vZ_{12}=0$ or $vZ_{12}\neq 0$.\\
\textbf{Case 1:} Suppose $vZ_{12}=0$. Then the action of $Z_{12}$ on $\mathcal{N}$ is trivial. Now $\mathcal{N}$ becomes a simple module over the factor algebra $\mathcal{M}\text{at}_2(q)/\langle Z_{12}\rangle$ which is isomorphic to a quantum affine space $\mathcal{O}_{\mathbf{q}}(\mathbb{K}^3)$ of rank $3$ under the correspondence $\overline{Z}_{11}\mapsto x_1, \overline{Z}_{21}\mapsto x_2, \ \overline{Z}_{22}\mapsto x_3$, where \[\mathbf{q}=\begin{pmatrix}
1&q^{-1}&1\\
q&1&1\\
1&1&1
\end{pmatrix}.\] 
Now $v\overline{\de_{q}}=v(\overline{Z}_{11}\overline{Z}_{22}-\overline{Z}_{12}\overline{Z}_{21})=\lambda_1v$ implies $vx_1x_3=\lambda_1v\neq 0$ in $\mathcal{O}_{\mathbf{q}}(\mathbb{K}^3)$. Thus in this case $\mathcal{N}$ is $x_1$ and $x_3$-torsion free simple $\mathcal{O}_{\mathbf{q}}(\mathbb{K}^3)$-module.\\
\textbf{Case 2:} Suppose $vZ_{12}\neq 0$. Then the action of $Z_{21}$ on $\mathcal{N}$ is trivial. Now $\mathcal{N}$ becomes a simple module over the factor algebra $\mathcal{M}\text{at}_2(q)/\langle Z_{21}\rangle$ which is isomorphic to a quantum affine space $\mathcal{O}_{\mathbf{q}}(\mathbb{K}^3)$ of rank $3$ under the correspondence $\overline{Z}_{11}\mapsto x_1, \overline{Z}_{12}\mapsto x_2, \ \overline{Z}_{22}\mapsto x_3$, where \[\mathbf{q}=\begin{pmatrix}
1&1&1\\
1&1&q^{-1}\\
1&q&1
\end{pmatrix}.\]  Similarly in this case $\mathcal{N}$ is $x_1$ and $x_3$-torsion free simple $\mathcal{O}_{\mathbf{q}}(\mathbb{K}^3)$-module.
\par In either case $\mathcal{N}$ is an  $x_1,x_3$-torsion free simple module over the quantum affine space $\mathcal{O}_{\mathbf{q}}(\mathbb{K}^3)$ rank $3$. The simple modules over such quantum affine spaces have been classified already in \cite{smsb}. Here the possible $\mathbb{K}$-dimensions of $\mathcal{N}$ is $1$ or $m$.
\subsection{Simple $\de_{q}$-torsion and $Z_{12}Z_{21}$-torsion free $\mathcal{M}\text{at}_2(q)$-modules:} If $\lambda_1=0$ and $\lambda_2\neq 0$, then the action of $\de_{q}$ on $\mathcal{N}$ is trivial. Hence $\mathcal{N}$ becomes a simple module over the factor algebra $\mathcal{M}\text{at}_2(q)/\langle \de_{q}\rangle$ which is isomorphic to the algebra $\mathcal{O}_{\mathbf{q}}(\mathbb{K}^4)/\langle x_1x_4-x_2x_3\rangle$ under the correspondence $\overline{Z}_{11}\mapsto x_1,\ \overline{Z}_{12}\mapsto x_2,\ \overline{Z}_{21}\mapsto x_3, \ \overline{Z}_{22}\mapsto x_4,$ where \[\mathbf{q}=\begin{pmatrix}
1&1&q^{-1}&q^{-1}\\
1&1&q^{-1}&q^{-1}\\
q&q&1&1\\
q&q&1&1
\end{pmatrix}.\]
Here $\lambda_2\neq 0$ implies $\mathcal{N}$ is $x_2$ and $x_3$-torsion free and hence from the relation $x_1x_4=x_2x_3$ we conclude that $\mathcal{N}$ is $x_1$ and $x_4$-torsion free also. Thus $\mathcal{N}$ is $x_i$-torsion free simple module over $\mathcal{O}_{\mathbf{q}}(\mathbb{K}^4)/\langle x_1x_4-x_2x_3\rangle$.  In this case one can classified simple modules $\mathcal{N}$ by taking common eigen vector of the commuting operators $x_1,x_2$ and the possible $\mathbb{K}$-dimension of $\mathcal{N}$ is $m$ only (cf. \cite{smsb}).

\subsection{Simple $\de_{q},Z_{12}Z_{21}$-torsion $\mathcal{M}\text{at}_2(q)$-modules:} If $\lambda_1=0$ and $\lambda_2=0$, then $\mathcal{N}$ becomes a simple module over the factor algebra \[\mathcal{M}\text{at}_2(q)/\langle \de_{q},Z_{12}Z_{21}\rangle\cong \mathcal{O}_{\mathbf{q}}(\mathbb{K}^4)/\langle x_1x_4,x_2x_3\rangle.\] Here the action of $x_1$ or $x_4$ on $\mathcal{N}$ is trivial and the action of $x_2$ or $x_3$ on $\mathcal{N}$ is trivial. In this case one can classified simple module $\mathcal{N}$ and the possible $\mathbb{K}$-dimensions of $\mathcal{N}$ are $1$ or $m$ (cf. \cite{smsb}).
\par In all the previous three observations simple $\mathcal{M}\text{at}_2(q)$-module collapses into a structure of a simple module over a much simpler algebra namely quantum affine space or its factor.
\subsection{Simple $\de_{q},Z_{12}Z_{21}$-torsionfree $\mathcal{M}\text{at}_2(q)$-modules:}\label{ssrf}
In view of paragraphs $(2.1)-(2.3)$, henceforth we can assume that $\lambda_1\neq 0,\lambda_2\neq 0$, that is, $\mathcal{N}$ is $\de_{q}$ and $Z_{12}Z_{21}$-torsion free simple $\mathcal{M}\text{at}_2(q)$-module. Suppose $q$ be a primitive $m$-th root of unity. Since each of the elements \begin{equation}\label{1}
    Z^m_{11},Z^m_{21},Z_{22}^m,\de_{q},Z_{12}Z_{21}
\end{equation} of $\mathcal{M}\text{at}_2(q)$ commutes, there is a common eigenvector $v$ in $\mathcal{N}$ corresponding the operators (\ref{1}). Put 
\[vZ^m_{11}=\alpha v,\ vZ^m_{21}=\beta v,\ vZ^m_{22}=\gamma v,\ v\de_{q}=\lambda_1v,\ vZ_{12}Z_{21}=\lambda_2v,\]
for some $\alpha,\beta\in \mathbb{K}$ and $\lambda_1,\lambda_2\in \mathbb{K}^*$. Note that the central elements  $Z^m_{11},Z^m_{21}$ and $Z_{22}^m$ act as multiplication by scalar on $\mathcal{N}$, by Schur's lemma. As $Z_{21}$ is a normal element and $\lambda_2\neq 0$, therefore we have $\beta\neq 0$. In the following we shall determine the structure of simple $\mathcal{M}\text{at}_2(q)$-module $\mathcal{N}$ according to the scalars:\\
\textbf{Case I:} Let us first consider $\alpha\neq 0$. Then the vectors $e(a,b):=vZ_{11}^aZ_{21}^b$ where $0 \leq a,b \leq m-1$ of $\mathcal{N}$ are non-zero, as $\alpha\neq 0,\ \beta\neq 0$. Consider the vector subspace $\mathcal{N}_1$ of $\mathcal{N}$ spanned by these vectors $e(a,b)$. Now we claim that $\mathcal{N}_1$ is invariant under the action of $\mathcal{M}\text{at}_2(q)$. In fact after some straightforward calculation using the defining relations of $\mathcal{M}\text{at}_2(q)$ and the identities in Lemma \ref{c1}, we have 
   \[\begin{array}{l}
    e(a,b)Z_{11}=
    \begin{cases}
    q^{b}e(a+1,b),& \text{when}\ \ a<m-1\\
    q^b\alpha e(0,b),&\text{when}\ \ a=m-1
    \end{cases}\\
    e(a,b)Z_{12}=\begin{cases}q^{(b-a)}\lambda_2e(a,b-1),& \text{when}\ \  b>0 \\
      \beta^{-1}\lambda_2q^{-a} e(a,m-1),& \text{when} \ \ b=0 
       \end{cases}\\
       e(a,b)Z_{21}=\begin{cases}e(a,b+1),& \text{when}\ \ b<m-1\\
       \beta e(a,0),& \text{when}\ \ b=m-1
       \end{cases}\\
       e(a,b)Z_{22}=\begin{cases} \left(\lambda_1+q\lambda_2-q(1-q^{-a})\lambda_2\right)e(a-1,b),& \text{when} \ \ a>0 \\
       \alpha^{-1}\left(\lambda_1+q\lambda_2\right)e(m-1,b),& \text{when} \ \ a=0
       \end{cases}
       \end{array}\]
       Therefore owing to simpleness of $\mathcal{N}$, $\mathcal{N}=\mathcal{N}_1$. The $\mathcal{M}\text{at}_2(q)$-module $\mathcal{N}_1$ given above is denoted by $(\mathcal{N}_1,\alpha,\beta,\lambda_1,\lambda_2)$. Now the following result deciphers the $\mathbb{K}$-dimension of $\mathcal{N}_1$.
\begin{theom}\label{t1}
The simple $\mathcal{M}\text{at}_2(q)$-module $(\mathcal{N}_1,\alpha,\beta,\lambda_1,\lambda_2)$ has dimension $m^{2}$.
\end{theom}
\begin{proof}
Here we will use induction on the $r$-elements subsets consisting of the non zero vectors $e(a,b)$. Note that for $e(a_1,b_1) \neq e(a_2,b_2)$, both of these vectors are eigen vectors of the operator det$_{q}$ or $Z_{12}Z_{21}$ on $\mathcal{N}_1$ with distinct eigen values. Indeed, if $b_1\neq b_2$, then $e(a_1,b_1)$ and $e(a_2,b_2)$ are eigen vectors of det$_{q}$ corresponding to the distinct eigen values $q^{b_1}\lambda_1$ and $q^{b_2}\lambda_1$ respectively. If $b_1 =b_2$ and $a_1 \neq a_2$ then $e(a_1,b_1)$ and $e(a_2,b_2)$ are eigen vectors of $Z_{12}Z_{21}$ corresponding to the distinct eigen values $q^{b_1-a_1}\lambda_2$ and $q^{b_2-a_2}\lambda_2$ respectively.
\par Let the result be true for any such $k$ vectors. Let $S:=\{e\left(a^{(i)},b^{(i)}\right):i=1,\cdots,k+1\}$ be a set of $k+1$ such vectors. Suppose  
\begin{equation}\label{sum}
  p:=\sum_{i=1}^{k+1} \zeta_i~e\left(a^{(i)},b^{(i)}\right)=0
\end{equation} for some $\zeta_i\in \mathbb{K}$. Now $e\left(a^{(k)},b^{(k)}\right) \neq e\left(a^{(k+1)},b^{(k+1)}\right)$. Without loss of generality let us assume both these vectors have distinct eigen values $\nu_k$ and $\nu_{k+1}$ respectively corresponding to the operator det$_{q}$. Now 
\[0=p\de_{q}-\nu_{k+1}p=\sum_{i=1}^{k} \zeta_i(\nu_i-\nu_{k+1})~e\left(a^{(i)},b^{(i)}\right).
\]
Using induction hypothesis,
\begin{center}
    $\zeta_i(\nu_i-\nu_{k+1})=0$ for all $i=1,\cdots,k$.
\end{center}
As $\nu_k \neq \nu_{k+1}$, $\zeta_k=0$.
Putting $\zeta_k=0$ in (\ref{sum}) and using induction hypothesis we have $\zeta_i=0$ for all $i=1,\cdots,k+1$. Hence the result follows.
\end{proof}
\begin{remak}
The simple $\mathcal{M}\text{at}_2(q)$-modules $(\mathcal{N}_1,\alpha,\beta,\lambda_1,\lambda_2)$ and $(\mathcal{N}_1,\alpha',\beta',\lambda'_1,\lambda'_2)$ are isomorphic if and only if $\alpha=\alpha',\beta=\beta',\lambda_1=q^b\lambda'_1$ and $\lambda_2=q^{b-a}\lambda'_2$, for $0\leq a,b\leq m-1$.
\end{remak}
~\\\textbf{Case II:} Next consider $\alpha=0$. That is, $vZ^m_{11}=0$. So there exists $0\leq r\leq m-1$ such that $vZ_{11}^{r}\neq 0$ and $vZ^{r+1}_{11}=0$. Define $u:=vZ_{11}^{r}\neq 0$. Then we compute that $uZ^m_{22}=\gamma u$. We now claim that the vectors $uZ^a_{22},\ 0\leq a\leq m-1$ are nonzero for any choice of $\gamma\in \mathbb{K}$. If $\gamma\neq 0$, then we are done. Otherwise if $\gamma=0$, suppose $s$ be the smallest integer with $1 \leq s \leq m$ such that $uZ_{22}^{s-1} \neq 0$ and $uZ_{22}^s=0$. Then after simplifying the equality $uZ_{22}^sZ_{11}=0$ we have
    \[0=uZ_{22}^sZ_{11}=u\Big(Z_{11}Z_{22}^s-(1-q^s)Z_{12}Z_{21}Z_{22}^{s-1}\Big)=(q^s-1)q^{-r}\lambda_2 \ uZ_{22}^{s-1}.\]
This implies $s$ is the smallest index with $1\leq s\leq m$ such that $q^s=1$. Thus we have $s=m$ and the claim follows. 
\par Now by the above claim along with the $\beta\neq 0$, it follows that the vectors $e(a,b):=uZ_{22}^aZ_{21}^b$, where $0 \leq a,b \leq m-1$ of $\mathcal{N}$ are non-zero. Set $\mathcal{N}_2$ be the vector subspace of $\mathcal{N}$ spanned by these non zero vectors $e(a,b)$. One can easily verify that $\mathcal{N}_2$ is $\mathcal{M}\text{at}_2(q)$-invariant. In fact after some straightforward computation we have the following:
\[\begin{array}{l}
        e(a,b)Z_{11}=\begin{cases} 
        q^{b}(q^{a}-1)q^{-r}\lambda_2\ e(a-1,b),&\text{when} \ \ a>0 \\
      0,&\text{when} \ \ a=0
       \end{cases}\\
       e(a,b)Z_{12}=\begin{cases} q^{(b+a)}q^{-r}\lambda_2e(a,b-1),&\text{when} \ \ b>0 \\
      \beta^{-1}q^{a}q^{-r}\lambda_2 e(a,m-1),& \text{when} \ \ b=0
       \end{cases}\\
       e(a,b)Z_{21}=\begin{cases}
       e(a,b+1),& \text{when} \ \ 0\leq b\leq m-2\\
       \beta e(a,0),& \text{when} \ \  b=m-1
       \end{cases}\\
       e(a,b)Z_{22}= \begin{cases}e(a+1,b),& \text{when} \ \ a < m-1\\
       \gamma e(0,b),&\text{when} \ \ a=m-1
       \end{cases}
       \end{array}\]
Therefore $\mathcal{N}$ being simple module, we have $\mathcal{N}=\mathcal{N}_2$. The $\mathcal{M}\text{at}_2(q)$-module $\mathcal{N}_2$ given above is denoted by $(\mathcal{N}_2,\beta,\gamma,\lambda_2)$. Now the following result ensures the $\mathbb{K}$-dimension of $\mathcal{N}_2$ and the proof is parallel to the Theorem \ref{t1}.
\begin{theom}
The simple $\mathcal{M}\text{at}_2(q)$-module $(\mathcal{N}_2,\beta,\gamma,\lambda_2)$ has dimension $m^{2}$.
\end{theom}
\begin{remak}
The simple $\mathcal{M}\text{at}_2(q)$-modules $(\mathcal{N}_2,\beta,\gamma,\lambda_2)$ and $(\mathcal{N}_2,\beta',\gamma',\lambda'_2)$ are isomorphic if and only if $\beta=\beta',\gamma=\gamma'$ and $\lambda_2=q^{a+b}\lambda'_2$, for $0\leq a,b\leq m-1$.
\end{remak}
\begin{remak}
It is clear from the action of $\mathcal{M}\text{at}_2(q)$ that there does not exist any isomorphism between the above two types of simple $\mathcal{M}\text{at}_2(q)$-modules. Indeed  no non zero element of $\mathcal{N}_1$ is annihilated by $Z_{11}$, but $e(0,b)$ in $\mathcal{N}_2$ is annihilated by $Z_{11}$.
\end{remak}
\par Finally the above discussions lead us to the main result of this section which provides an opportunity for classification of simple $\mathcal{M}\text{at}_2(q)$-modules in terms of scalar parameters.
\begin{theom}
Let $q$ be a primitive $m$-th root of unity. Then each simple $\de_{q},Z_{12}Z_{21}$-torsionfree $\mathcal{M}\text{at}_2(q)$-module is isomorphic to one of the following simple $\mathcal{M}\text{at}_2(q)$-modules:
\begin{enumerate}
    \item $(\mathcal{N}_1,\alpha,\beta,\lambda_1,\lambda_2)$ for some $(\alpha,\beta,\lambda_1,\lambda_2)\in(\mathbb{K}^*)^4$
    \item $(\mathcal{N}_2,\beta,\gamma,\lambda_2)$ for some $(\beta,\lambda_2)\in(\mathbb{K}^*)^2$ and $\gamma\in\mathbb{K}$.
\end{enumerate}
\end{theom}
Thus in view of subsections (\ref{ss1})-(\ref{ssrf}), one can conclude that the simple $\de_{q},Z_{12}Z_{21}$-torsion free $\mathcal{M}\text{at}_2(q)$-modules have only maximal $\mathbb{K}$-dimension which is equal to $\pideg (\mathcal{M}\text{at}_2(q))$.

\section{\bf{Finite Dimensional Indecomposable $\mathcal{M}\text{at}_2(q)$-Modules}}
In this section we aim to construct some finite dimensional indecomposable modules over $\mathcal{M}\text{at}_2(q)$. Let $q$ be a primitive $m$-th root of unity. Let $\mathcal{B}$ be the subalgebra of $\mathcal{M}\text{at}_2(q)$ generated by the monomials $Z_{11},Z_{12}$ and $Z_{21}^{m}$. Observe that $\mathcal{B}$ is a commutative polynomial algebra $\mathbb{K}[Z_{11},Z_{12},Z^m_{21}]$. Then for $\underline{\lambda}:=(\lambda_1,\lambda_2)\in(\mathbb{K}^*)^2$, there is a one dimensional $\mathcal{B}$-module $\mathbb{K}\cong\mathbb{K}(\underline{\lambda})$ given by
\[vZ_{11}=0,\ vZ_{12}=\lambda_1v,\ vZ_{21}^{m}=\lambda_2v\]
for $v\in \mathbb{K}$. Define the Verma module $M(\underline{\lambda})$ by
\[M(\underline{\lambda}):=\mathbb{K}(\underline{\lambda})\otimes_{\mathcal{B}}\mathcal{M}\text{at}_2(q).\]
It follows from Proposition \ref{imp} that $\mathcal{M}\text{at}_2(q)$ is a free left $\mathcal{B}$-module with basis \[\{Z_{21}^aZ_{22}^b~|~0\leq a\leq m-1, b\geq 0\},\]
and therefore $M(\underline{\lambda})$ has a vector space basis \[f(a,b):=v \otimes Z_{21}^aZ_{22}^b,\ \ 0\leq a\leq m-1, b\geq 0.\]
Thus $M(\underline{\lambda})$ carries an infinite dimensional right $\mathcal{M}\text{at}_2(q)$-module. The explicit action of the generators of $\mathcal{M}\text{at}_2(q)$ on $M(\underline{\lambda})$ is given by
\[
\begin{array}{l}
f(a,b)Z_{11}=\begin{cases}
0,&b=0\\
\lambda_1q^a(q^b-1)f(a+1,b-1),& b>0
\end{cases}\\
f(a,b)Z_{12}=q^{a+b}\lambda_1f(a,b)\\
f(a,b)Z_{21}=\begin{cases}
f(a+1,b),& 0\leq a \leq m-2\\
\lambda_2f(0,b),& a=m-1
\end{cases}\\
f(a,b)Z_{22}=f(a,b+1)
\end{array}\]
As $q$ is a primitive $m$-th root of unity, then one gets the invariant subspaces
\[M_{p,m}:=\spa\{f(a,pm+i)~|~i\geq 0,0\leq a\leq m-1\},\ \ \ \ p=1,2,\cdots.\]
Evidently one has
\[M(\underline{\lambda})\supset M_{1,m}\supset M_{2,m} \supset M_{3,m}\supset\cdots\]
and \[Q_{1,m}\subset Q_{2,m}\subset Q_{3,m}\subset \cdots,\]
where $Q_{p,m}=M(\underline{\lambda})/M_{p,m}$ is the quotient space corresponding to $M_{p,m}$. Also observe that
\[Q_{p,m}=\spa\{\overline{f}(a,b)=f(a,b) \mo M_{p,m}~|~0\leq a\leq m-1,0\leq b\leq pm-1\}.\]
On each $Q_{p,m}$, the above action induces $pm^2$-dimensional $\mathcal{M}\text{at}_2(q)$-module structure:

\[
\begin{array}{l}
\overline{f}(a,b)Z_{11}=\begin{cases}
0,&b=0\\
\lambda_1 q^a(q^b-1)\overline{f}(a+1,b-1),& 1\leq b\leq pm-1
\end{cases}\\
\overline{f}(a,b)Z_{12}=q^{a+b}\lambda_1\overline{f}(a,b)\\
\overline{f}(a,b)Z_{21}=\begin{cases}
\overline{f}(a+1,b),& 0\leq a \leq m-2\\
\lambda_2\overline{f}(0,b),& a=m-1
\end{cases}\\
\overline{f}(a,b)Z_{22}=\begin{cases}
\overline{f}(a,b+1),&0\leq b\leq pm-2\\
0,& b=pm-1
\end{cases}
\end{array}\]
The following result provides all submodules of $Q_{p,m}$. 
\begin{lemm}\label{submod}
The $\mathcal{M}\text{at}_2(q)$-submodules of $M(\underline{\lambda})$ containing $M_{p,m}$ are of the form $M_{r,m}, 1\leq r\leq p$ or $M(\underline{\lambda})$.
\end{lemm}
\begin{proof}
Let $W$ be an $\mathcal{M}\text{at}_2(q)$-submodule containing $M_{p,m}$. If $W=M_{p,m}$, then we are done. Otherwise one can choose $0\neq w\in W$ such that
\[w=\sum_{i=0}^{m-1}\sum_{j=0}^{pm-1}C_{ij}f(i,j),\ \ C_{ij}\in \mathbb{K},\] with at least one non zero $C_{ij}$. Let $j'=\mi\{j~|~C_{ij}\neq 0\}$. Then $wZ_{22}^{pm-1-j'}\in W$. Now
\[
\begin{array}{cl}
wZ_{22}^{pm-1-j'}&=\sum\limits_{i=0}^{m-1}\sum\limits_{j=j'}^{pm-1}C_{ij}f(i,j+pm-1-j')\\
&=w'+\sum\limits_{i=0}^{m-1}C_{i,j'}f(i,pm-1),
\end{array}\]
where $w'=\sum\limits_{i=0}^{m-1}\sum\limits_{j=j'+1}^{pm-1}C_{ij}f(i,j+pm-1-j') \in M_{p,m}$. Thus \[w''=wZ_{22}^{pm-1-j'}-w'=\sum\limits_{i=0}^{m-1}C_{i,j'}f(i,pm-1)\in W.\] Now if $C_{i_1,j'}$ and $C_{i_2,j'}$ are two non zero scalars in $w''$, then $w''Z_{12}-\lambda_1q^{i_1-pm+1}w''$ is a non zero element in $W$ smaller length than $w''$. Hence by
induction it follows that $f(i',pm-1)\in W$ for some $0\leq i'\leq m-1$. Finally with the action of $Z_{11}$ and $Z_{21}$ on $f(i',pm-1)$, we have $M_{p-1,l}\subseteq W$. Now if $W=M_{p-1,l}$, then we are done. Otherwise continuing with the above argument sequentially one can obtain the desired result.
\end{proof}
With this lemma we have the following:
\begin{theo}\label{itre}
Suppose $q$ is a primitive $m$-th root of unity. Then
\begin{enumerate}
    \item [(1)] $Q_{1,m}$ is a simple $\mathcal{M}\text{at}_2(q)$-module.
    \item [(2)] When $p>1$, then $Q_{p,m}$ is neither simple nor semisimple $\mathcal{M}\text{at}_2(q)$-module.
    \item [(3)] When $p>1$, then $Q_{p,m}$ is an indecomposable $\mathcal{M}\text{at}_2(q)$-module.
\end{enumerate}
\end{theo}
\begin{proof}
Proof of (1): It follows from Lemma \ref{submod}.\\
Proof of (2): First of all $Q_{p,m}$ is semisimple if and only if for each invariant subspace $W\subset Q_{p,m}$, there is an invariant subspace $\overline{W}$ complementary to it. Now let us consider the subspace
\[W=\spa\{\overline{f}(a,b)~|~0\leq a\leq m-1,(p-1)m\leq b\leq pm-1\}.\]
It is easily seen that $W$ is $\mathcal{M}\text{at}_2(q)$-invariant. One can prove that it does not have an invariant complementary subspace $\overline{W}$ in $Q_{p,m}$. Otherwise a non zero element $x\in \overline{W}$ can be written as
\[x=\sum\limits_{i=0}^{m-1}\sum\limits_{j=0}^{(p-1)m-1}C_{ij}\overline{f}(i,j)+\sum\limits_{i=0}^{m-1}\sum\limits_{j=(p-1)m}^{pm-1}D_{ij}\overline{f}(i,j),\] where there is at least one non zero $C_{ij}$. Let $j'=\mi\{j~|~C_{ij}\neq 0\}$. Then we have $xZ_{22}^{(p-1)m-j'}\in W$. This contradicts the assumption that $W$ is $\mathcal{M}\text{at}_2(q)$-invariant.\\
Proof of (3): Recall that a module is indecomposable if it is non zero and cannot be written as a direct sum of two non zero submodules. By Lemma \ref{submod}, the non zero proper $\mathcal{M}\text{at}_2(q)$-invariant subspaces of $Q_{p,m}$ are of the form $M_{r,m}/M_{p,m},1\leq r\leq p-1$. Now with the argument in the proof of (2), one can easily verify that there is no invariant subspace in $Q_{p,m}$ complementary to any of the non zero proper $\mathcal{M}\text{at}_2(q)$-invariant subspaces. Thus when $p>1$, the module $Q_{p,m}$ is a finite dimensional indecomposable $\mathcal{M}\text{at}_2(q)$-module. 
\end{proof}
\begin{rema}
The simple $\mathcal{M}\text{at}_2(q)$-module $(Q_{1,m},\lambda_1,\lambda_2)$ is isomorphic to the simple module $(\mathcal{N}_2,\xi,0,\eta)$ where $\xi=\lambda_2,\eta^m=\lambda^m_1\lambda_2$. The indecomposable $\mathcal{M}\text{at}_2(q)$-modules $(Q_{p,m},\lambda_1,\lambda_2)$ and $(Q_{p,m},\lambda'_1,\lambda'_2)$ with $p>1$ are isomorphic if and only if $\lambda_1=q^{a+b}\lambda'_1$ and $\lambda_2=\lambda'_2$ for $0\leq a,b\leq m-1$. Also if $p_1\neq p_2$, then the indecomposable modules $Q_{p_1,m}$ and $Q_{p_1,m}$ are non-isomorphic. 
\end{rema}
\section{\bf{Reflection Equation Algebra $A_q(M_2)$}}\label{reas}
Recall the reflection equation algebra $A_q(M_2)$ of rank $2$. Assume that $q$ be a primitive $m$-th root of unity.
In this section we recall some facts for $A_q(M_2)$ that shall be applying to study simple module. Firstly the following result provides important ring theoretic properties for $A_q(M_2)$:
\begin{prop}\emph{(\cite[Proposition 3.1]{dl})}\label{imp1}
The algebra $A_q(M_2)$ is an affine noetherian domain and the monomials of the form $u_{11}^au_{12}^bu_{21}^cu_{22}^d$ constitute a $\mathbb{K}$-basis of $A_q(M_2)$.
\end{prop}
With this $\mathbb{K}$-basis one can show that $A_q(M_2)$ has an iterated skew polynomial presentation of the form:
\[\mathbb{K}[u_{11},u_{22}][u_{12},\sigma][u_{21},\tau,\delta]\] where the $\sigma$ and $\tau$ are $\mathbb{K}$-linear automorphisms and the $\delta$ is $\mathbb{K}$-linear $\tau$-derivation such that
\[\sigma(u_{11})=u_{11}+q^{-2}(1-q^{-2})u_{22},\ \sigma(u_{22})=q^{-2}u_{22}\]
\[\tau(u_{11})=u_{11}+(q^{-2}-1)u_{22},\ \tau(u_{22})=q^{2}u_{22},\ \tau(u_{12})=u_{12}\]
\[\delta(u_{11})=\delta(u_{22})=0,\ \delta(u_{12})=(1-q^{-2})(u_{11}-u_{22})u_{22}.\]
Let us consider the two elements \[\de_{q}:=u_{11}u_{22}-q^2u_{12}u_{21}\ \text{and}\ \tr_{q}:=u_{11}+q^{-2}u_{22}\] of the algebra $A_q(M_2)$ commonly known as quantum determinant and quantum trace respectively. Using defining relations of the algebra $A_q(M_2)$, one can easily verify the following:
\begin{itemize}
    \item[(i)] The elements $\de_q$ and $\tr_q$ are central in $A_q(M_2)$ (cf. \cite{cgj}).
    \item[(ii)] The element $u_{22}$ is a normal element of $A_q(M_2)$.
    \item[(iii)] The elements $u_{11},u_{22}$ and $u_{12}u_{21}$ commuting monomials in $A_q(M_2)$.
\end{itemize}
\begin{lemm}
\emph{(\cite[Lemma 3.2]{cgj})}\label{li1}
For $r\geq 1$, the following identities hold in the algebra $A_q(M_2)$:
\begin{enumerate}
\item[(i)] $u_{12}^ru_{11}=u_{11}u_{12}^r+q^{-2}(q^{2r}-1)u_{12}^ru_{22}$
    \item [(ii)] $u_{21}^ru_{11}=u_{11}u_{21}^r+q^{-2}(1-q^{2r})u_{22}u_{21}^r$
    \item [(iii)] $u_{21}^ru_{12}=u_{12}u_{21}^r+q^{-2}(q^{2r}-1)u_{11}u_{22}u_{21}^{r-1}+(1-q^{2r})q^{-2}u_{22}^2u_{21}^{r-1}$
    \item [(iv)] $u_{21}u^r_{12}=u^r_{12}u_{21}+(1-q^{-2r})u_{11}u_{22}u^{r-1}_{12}+q^{-4}(1-q^{4r}+q^{4r-2}-q^{2r-2})u_{12}^{r-1}u_{22}^2$
\end{enumerate}
\end{lemm}
The equalities can be proved by induction on $r$.
\par Let $q$ be a primitive $m$-th root of unity. Set $n$ such that 
\[n:=\begin{cases}
m,& \text{when} \ \  m \  \text{is odd}\\
\frac{m}{2},& \text{when} \ \  m \  \text{is even.}
\end{cases}\]
The identities listed in Lemma \ref{li1} along with the defining relations of $A_q(M_2)$ yield the following.  
\begin{coro}\emph{(\cite[Lemma 3.2]{cgj})}\label{li1}
If $q$ is a primitive $m$-th root of unity, then $u_{12}^n,u_{21}^n$ and $u_{22}^n$ are central elements in $A_q(M_2)$. We note that the element $u_{11}^n$ in $A_q(M_2)$ is not central.
\end{coro}
\begin{prop} \label{finite1}
The algebra $A_q(M_2)$ is a PI algebra if and only if $q$ is a root of unity.
\end{prop}
\begin{proof}
Suppose $q$ be a primitive $m$-th root of unity. Let $Z$ be the subalgebra of $A_q(M_2)$ generated by $u_{12}^n,u_{21}^n,u_{22}^n$ and $\tr_q$. Then $Z$ is a central subalgebra. Now by Proposition \ref{imp1} and relation $u_{11}=\tr_q-q^{-2}u_{22}$, each element of $A_q(M_2)$ can be expressed as $\mathbb{K}$ linear combination of $\tr^*_{q}u^*_{22}u^*_{12}u^*_{21}$ . Then one can easily verify that $A_q(M_2)$ is a finitely generated module over this central subalgebra $Z$ with basis \[\{u_{12}^bu_{21}^cu_{22}^d~|~0\leq b,c,d\leq n-1\}.\] Hence it follows from Proposition \ref{f} that $A_q(M_2)$ is PI algebra.
\par For the converse, just note that the $\mathbb{K}$-subalgebra of $A_q(M_2)$ generated by $u_{22}$ and $u_{12}$ with relation $u_{22}u_{12}=q^2u_{12}u_{22}$ is not PI if $q$ is not a root of unity (cf. \cite[Proposition I.14.2.]{brg}).
\end{proof}
The algebra $A_q(M_2)$ being a prime affine PI algebra over an algebraically closed field $\mathbb{K}$, it follows from Proposition \ref{sim} that the $\mathbb{K}$-dimension of each simple $A_q(M_2)$-module is finite and bounded above by the $\pideg(A_q(M_2))$. Infact this upper bound is attained (cf. \cite[Lemma III.1.2]{brg}). Now observe that the existing $\mathbb{K}$-automorphisms in the iterated skew polynomial presentation of $A_q(M_2)$ is much more complicated than quantized matrix algebra $\mathcal{M}\text{at}_n(q)$. So here we can not apply the key technique as in subsection \ref{pisub} to compute the PI degree of $A_q(M_2)$. In the forthcoming section, we shall focus on the structure of simple $A_q(M_2)$-modules. This will help us to obtain the invariant $\pideg(A_q(M_2))$ explicitly.
\section{\bf{Simple Modules over $A_q(M_2)$}}
Let $q$ be a primitive $m$-th root of unity and $N$ be a simple module over $A_q(M_2)$. Then $N$ is finite dimensional $\mathbb{K}$-space. Since the element $u_{22}$ is normal in $ A_q(M_2)$, the action of $u_{22}$ on $N$ is either trivial or invertible.
\subsection{{Simple $u_{22}$-torsion $A_q(M_2)$-modules}}\label{sub2}
Suppose the action of $u_{22}$ on $N$ is trivial. Then $N$ becomes a simple module over the factor algebra $A_q(M_2)/\langle u_{22}\rangle$ which is isomorphic to a commutative polynomial algebra $\mathbb{K}[x_1,x_2,x_3]$ under the correspondence $\overline{u}_{11}\mapsto x_1, \overline{u}_{12}\mapsto x_2, \ \overline{u}_{21}\mapsto x_3$. In this case the possible $\mathbb{K}$-dimension of $N$ is $1$ only.  
\subsection{Simple $u_{22}$-torsionfree $A_q(M_2)$-modules}\label{sub3} Suppose the action of $u_{22}$ on $N$ is invertible. Since each of the monomials \begin{equation}\label{cop}
    u^n_{12},u^n_{21},u_{11},u_{22},u_{12}u_{21}
\end{equation} of $A_q(M_2)$ commutes, there is a common eigenvector $v$ in $N$ corresponding to the operators (\ref{cop}). Put \[vu^n_{12}=\alpha v,~vu^n_{21}=\beta v,~vu_{11}=\lambda_1v,~vu_{22}=\lambda_2v,~vu_{12}u_{21}=\lambda_3v,\] for some $\alpha,\beta,\lambda_1,\lambda_3\in \mathbb{K}$ and $\lambda_2\in \mathbb{K}^*$.
By Schur's lemma, the central elements $u_{12}^n$ and $u_{21}^n$ act as multiplication by scalar on $N$. In the following we shall determine the structure of simple $A_q(M_2)$-module ${N}$ according to the scalars:\\
\textbf{Case I:} Let us assume $\beta\neq 0$. Then $vu^n_{21}\neq 0$ and so the vectors $vu_{21}^r$ where $0 \leq r \leq n-1$ of $N$ are non-zero. Let $N_1$ be the vector subspace of $N$ spanned by these non-zero vectors. One can prove that that $N_1$ is $A_q(M_2)$-invariant subspace of $N$. In fact after some straightforward calculation using the defining relations of $A_q(M_2)$ and the identities in Lemma \ref{li1}, we get
\[\begin{array}{l}
        (vu_{21}^r)u_{11}=(\lambda_1+q^{-2}(1-q^{2r})\lambda_2)vu_{21}^r\\
       (vu_{21}^r)u_{12}=\begin{cases}c_rvu_{21}^{r-1},&1\leq r\leq n-1\\
       \beta^{-1}\lambda_3vu_{21}^{n-1},&r=0
       \end{cases}\\
       (vu_{21}^r)u_{21}=\begin{cases}
       vu_{21}^{r+1},&0\leq r\leq n-2\\
       \beta v,& r=n-1
       \end{cases}\\
       (vu_{21}^r)u_{22}=q^{2r}\lambda_2vu_{21}^r
       \end{array}\] where \[c_r=\lambda_3+q^{-2}(q^{2r}-1)\lambda_1\lambda_2+(q^{2r}-q^{4r})q^{-4}\lambda_2^2,\ \ 1\leq r\leq n-1.\]
Since $N$ is simple $A_q(M_2)$-module, we have $N=N_1$. Note that the vectors $\{vu_{21}^r|0\leq r\leq n-1\}$ are linearly independent because these vectors are eigen vectors of the operator $u_{22}$ with distinct eigen values. Thus ${N}_1$ is a simple $A_q(M_2)$-module of dimension $n$.\\ 
\textbf{Case II:} Let us consider $\alpha \neq 0$ and $\beta=0$. Then the vector subspace $N_2$ of $N$ spanned by the non zero vectors $vu^r_{12}$ with $0\leq r\leq n-1$ is $A_q(M_2)$-invariant subspace. This verification is similar to the Case I. Also in this case $N_2$ is a simple $A_q(M_2)$-module of dimension $n$.\\
\textbf{Case III:} Let us consider $\alpha=\beta=0$. Then the operators $u_{12}$ and $u_{21}$ are nilpotent and hence $\ke (u_{12}):=\{x\in N~|~xu_{12}=0\}$ and $\ke(u_{21}):=\{y\in N~|~yu_{21}=0\}$ are non-zero subspaces of $N$. Therefore there is a common eigenvector $w$ in $\ke(u_{12})$ corresponding to the commuting operators (\ref{cop}). Take \[wu_{12}=0,~wu^n_{21}=0,~wu_{11}=\lambda'_1w,~wu_{22}=\lambda'_2w,~wu_{12}u_{21}=0,\] for some $\lambda'_1,\lambda'_2\in \mathbb{K}$. Clearly $\lambda'_2\neq 0$, because the action of $u_{22}$ on $N$ is invertible. Next consider the sequence of vectors of the operator $u_{21}$ as below
\[w,wu_{21},wu_{21}^2,\cdots,wu_{21}^{n-1},wu_{21}^n=0.\] Let $s$ be the smallest integer with $1 \leq s \leq n$ such that $wu_{21}^{s-1} \neq 0$ and $wu_{21}^s=0$. Now we claim that either $s=n$ or $s$ satisfies the relation $\lambda'_1=q^{2s-2}\lambda'_2$. Indeed after simplifying the equality $wu^{s}_{21}u_{12}=0$ we obtain
\[0=wu^{s}_{21}u_{12}=q^{-2}(q^{2s}-1)(\lambda'_1-q^{2s-2}\lambda'_2)\lambda'_2wu^{s-1}_{21}.\]
This implies either $q^{2s}=1$ or $\lambda'_1=q^{2s-2}\lambda'_2$. Thus the claim follows.
\par Now for such choice of $1\leq s\leq n$, let $N_3$ be the vector subspace of $N$ spanned by the non zero vectors $wu_{21}^r$ where $0 \leq r \leq s-1$. Then one can prove that $A_q(M_2)$ stabilizes the vector space $N_3$. In fact after some direct calculation, we get
\[\begin{array}{l}
        (wu_{21}^r)u_{11}=(\lambda'_1+q^{-2}(1-q^{2r})\lambda'_2)wu_{21}^r\\
       (wu_{21}^r)u_{12}=\begin{cases}c_rwu_{21}^{r-1},&1\leq r\leq s-1\\
       0,&r=0
       \end{cases}\\
       (wu_{21}^r)u_{21}=\begin{cases}
       wu_{21}^{r+1},&0\leq r\leq s-2\\
       0,& r=n-1
       \end{cases}\\
       (wu_{21}^r)u_{22}=q^{2r}\lambda'_2wu_{21}^r
       \end{array}\] where \[c_r=q^{-2}(q^{2r}-1)\lambda'_1\lambda'_2+(q^{2r}-q^{4r})q^{-4}(\lambda'_2)^{2},\ \ 1\leq r\leq s-1.\]
Therefore owing to simpleness of $N$, $N=N_3$. Note that the vectors $\{wu_{21}^r~|~0\leq r\leq s-1\}$ are eigen vectors of the operator $u_{22}$ corresponding to distinct eigen values. Therefore these vectors are linearly independent. Thus in this case $N_3$ is a simple $A_q(M_2)$-module of dimension $s$ if $\lambda'_1=q^{2s-2}\lambda'_2$ for some $1\leq s\leq n-1$ and otherwise simple module of dimension $n$.

\begin{remak}
It is clear from the action of $A_q(M_2)$ that there does not exist any isomorphism between the above three types of simple $A_q(M_2)$-modules. Indeed 
\begin{itemize}
    \item[(1)] no non zero element of $N_1$ is annihilated by $u_{21}$, but there are some nonzero vector in $N_2$ or in $N_3$  annihilated by $u_{21}$, and
    \item[(2)] no non zero element of $N_2$ annihilated by $u_{12}$, but $w$ in $N_3$ is annihilated by $u_{12}$.
\end{itemize} 
\end{remak}
Finally the above discussions lead us to the main result of this section:
\begin{theom}
Let $q$ be a primitive $m$-th root of unity. Then each simple $u_{22}$-torsionfree $A_q(M_2)$-module is isomorphic to one of the simple $A_q(M_2)$-modules $N_i$ for some $i=1,2,3$ as mentioned above.
\end{theom}
\subsection{PI degree of $A_q(M_2)$} So far we have classified all simple $A_q(M_2)$-modules in the subsections (\ref{sub2}) and (\ref{sub3}). Observe that the maximal $\mathbb{K}$-dimensions of simple $A_q(M_2)$-modules is $n$. Now the algebra $A_q(M_2)$ being a prime affine PI algebra over an algebraically closed field $\mathbb{K}$, it follows from Proposition \ref{sim} along with \cite[Lemma III.1.2]{brg} that the exact value of $\pideg A_q(M_2)$ is $n$. Thus here we have explicitly determine the PI degree of the algebra $A_q(M_2)$.
\section{\bf{Finite Dimensional Indecomposable $A_q(M_2)$-Modules}}
In this section we aim to construct some finite dimensional indecomposable modules over $A_q(M_2)$. Let $q$ be a primitive $m$-th root of unity. Let $\mathcal{B}$ be the subalgebra of $A_q(M_2)$ generated by the $u_{11},u_{12}$ and $u_{22}$. Observe that $\mathcal{B}u_{12}=u_{12}\mathcal{B}$ is an ideal of $\mathcal{B}$ and $\mathcal{B}/\mathcal{B}u_{12}\cong\mathbb{K}[u_{11},u_{22}]$ is a commutative algebra. Then for $\underline{\lambda}:=(\lambda_1,\lambda_2)\in(\mathbb{K}^*)^2$, there is a one dimensional $\mathcal{B}$-module $\mathbb{K}(\underline{\lambda})$ given by
\[vu_{12}=0,\ vu_{11}=\lambda_1v,\ vu_{22}=\lambda_2v\]
for $v\in \mathbb{K}=\mathbb{K}(\underline{\lambda})$. Define the Verma module $M(\underline{\lambda})$ by
\[M(\underline{\lambda}):=\mathbb{K}(\underline{\lambda})\otimes_{\mathcal{B}}A_q(M_2).\]
Since $A_q(M_2)$ is a free left $\mathcal{B}$-module with basis $\{u_{21}^r|r\geq 0\}$ by Proposition \ref{imp1}, $M(\underline{\lambda})$ has a vector space basis $f(r):=v \otimes u_{21}^r,\ r\geq 0$.
Thus $M(\underline{\lambda})$ carries an infinite dimensional right $A_q(M_2)$-module. The explicit action of the generators of $A_q(M_2)$ on $M(\underline{\lambda})$ is given by
\[
\begin{array}{l}
f(r)u_{11}=(\lambda_1+q^{-2}(1-q^{2r})\lambda_2)f(r)\\
f(r)u_{12}=\begin{cases}
0,&r=0\\
c_rf(r-1),&r\geq 1
\end{cases}\\
f(r)u_{21}=f(r+1)\\
f(r)u_{22}=q^{2r}\lambda_2f(r),
\end{array}\]
where $c_r=q^{-2}(q^{2r}-1)\lambda_1\lambda_2+(q^{2r}-q^{4r})q^{-4}\lambda_2^{2},\ r\geq 1.$
\par Let us set \[n=\begin{cases}
m,& \text{if}\ m \ \text{is odd}\\
\frac{m}{2},& \text{if}\ m\ \text{is even.}
\end{cases}\]  Then owing to the equation $c_{pn}=0$, one gets the $A_q(M_2)$-invariant subspaces of $M(\underline{\lambda})$
\[S_{p,n}:=\spa\{f(pn+i)~|~i\geq 0\},\ \ \ \ p=1,2,\cdots.\]
Now observe that \[M(\underline{\lambda})\supset S_{1,n}\supset S_{2,n} \supset S_{3,n}\supset\cdots.\]
Let $Q_{p,n}:=M(\underline{\lambda})/S_{p,n}$ be the quotient space corresponding to $S_{p,n}$. Then we can write
\[Q_{p,n}=\spa\{\overline{f}(r)=f(r) \mo S_{p,n}~|~0\leq r\leq pn-1\}.\]
The above action induces $pn$-dimensional $A_q(M_2)$-module structure on each $Q_{p,n}$. Next we shall focus on the $A_q(M_2)$-submodules of $Q_{p,n}$. Similar argument as in Lemma \ref{submod}, one can show that the non-zero proper $A_q(M_2)$-submodules of $M(\underline{\lambda})$ containing $S_{p,n}$ are of the form $S_{l,n}, 1\leq l\leq p-1$. Now with this we have the following:
\begin{theo}
Suppose $q$ is a primitive $m$-th root of unity. Then
\begin{enumerate}
    \item [(1)] $Q_{1,n}$ is a simple $A_q(M_2)$-module.
    \item [(2)] When $p>1$, then $Q_{p,n}$ is neither simple nor semisimple $A_q(M_2)$-module.
    \item [(3)] When $p>1$, then $Q_{p,n}$ is an indecomposable $A_q(M_2)$-module.
\end{enumerate}
\end{theo}
The proof of this theorem is parallel to the Theorem \ref{itre}.

\end{document}